\documentclass[a4paper,12pt,reqno]{amsart}
\pdfoutput=1

\usepackage[T1]{fontenc}      
\usepackage{comment}
\usepackage{siunitx}
\usepackage{parskip}
\usepackage[dvipsnames]{xcolor} 
\usepackage[british]{babel}   
\usepackage{csquotes}         
\usepackage[babel]{microtype} 
\usepackage{mathtools}        
\usepackage{amssymb}          
\usepackage{amsthm}           
\usepackage{thm-restate}    
\usepackage{amstext}          
\usepackage{mleftright}       
\usepackage{gensymb}          
\usepackage[makeroom]{cancel} 
\usepackage{mathdots}         
\usepackage{dsfont}           
\usepackage{stickstootext}
\usepackage[stickstoo,varbb]{newtxmath}
\makeatletter
\DeclareFontFamily{U}{ntxmia}{\skewchar \font =127}
 \DeclareFontShape{U}{ntxmia}{m}{it}{
                        <-> \ntxmath@scaled ntxmia
                      }{}    
                      \DeclareFontShape{U}{ntxmia}{b}{it}{
                        <-> \ntxmath@scaled ntxbmia
                      }{}
\makeatother

\usepackage{graphicx}         
\usepackage{psfrag}           
\usepackage{caption}          
\usepackage{tikz}             
\usetikzlibrary{backgrounds}

\usepackage{booktabs}         

\usepackage{enumitem}         

\usepackage[margin=1in]{geometry} 

\usepackage[textsize=footnotesize]{todonotes}

\usepackage[numbers,sort]{natbib}

\makeatletter
\def\NAT@spacechar{~}
\makeatother
\newcommand{\urlprefix}{}

\usepackage{hyperref}              
\usepackage[capitalise,nameinlink]{cleveref}  
\hypersetup{colorlinks=true,
   citecolor=green!60!black,
   filecolor=blue,
   linkcolor=blue!60!black,
   urlcolor=blue!60!black
}
\usepackage{url}

\makeatletter
\if@cref@capitalise
\crefname{figure}{figure}{figures}
\crefname{claim}{Claim}{Claims}
\crefname{conjecture}{Conjecture}{Conjectures}
\else
\crefname{figure}{Figure}{Figures}
\crefname{claim}{claim}{claims}
\crefname{conjecture}{conjecture}{conjectures}
\fi
\makeatother
\Crefname{figure}{Figure}{Figures}
\Crefname{claim}{Claim}{Claims}
\crefname{conjecture}{Conjecture}{Conjectures}

\usepackage{algorithm}
\usepackage{algpseudocode}

\allowdisplaybreaks

\newtheorem{theorem}{Theorem}

\newtheorem{lemma}[theorem]{Lemma}

\theoremstyle{definition}

\numberwithin{equation}{section}

\renewcommand{\binom}[2]{\ensuremath{\mleft(\kern-.1em\genfrac{}{}{0pt}{}{#1}{#2}\kern-.1em\mright)}}    
\newcommand{\inbinom}[2]{\ensuremath{\bigl(\kern-.1em\genfrac{}{}{0pt}{}{#1}{#2}\kern-.1em\bigr)}} 

\newcommand{\cC}{\mathcal{C}}

\newcommand{\cF}{\mathcal{F}}

\newcommand{\cP}{\mathcal{P}}
\newcommand{\cQ}{\mathcal{Q}}

\newcommand{\cT}{\mathcal{T}}

\DeclareSymbolFont{sfletters}{OML}{cmbrm}{m}{it}
\DeclareMathSymbol{\salpha}{\mathord}{sfletters}{"0B}
\DeclareMathSymbol{\sbeta}{\mathord}{sfletters}{"0C}
\DeclareMathSymbol{\sgamma}{\mathord}{sfletters}{"0D}
\DeclareMathSymbol{\sdelta}{\mathord}{sfletters}{"0E}
\DeclareMathSymbol{\sepsilon}{\mathord}{sfletters}{"0F}
\DeclareMathSymbol{\szeta}{\mathord}{sfletters}{"10}
\DeclareMathSymbol{\seta}{\mathord}{sfletters}{"11}
\DeclareMathSymbol{\stheta}{\mathord}{sfletters}{"12}
\DeclareMathSymbol{\siota}{\mathord}{sfletters}{"13}
\DeclareMathSymbol{\skappa}{\mathord}{sfletters}{"14}
\DeclareMathSymbol{\slambda}{\mathord}{sfletters}{"15}
\DeclareMathSymbol{\smu}{\mathord}{sfletters}{"16}
\DeclareMathSymbol{\snu}{\mathord}{sfletters}{"17}
\DeclareMathSymbol{\sxi}{\mathord}{sfletters}{"18}
\DeclareMathSymbol{\spi}{\mathord}{sfletters}{"19}
\DeclareMathSymbol{\srho}{\mathord}{sfletters}{"1A}
\DeclareMathSymbol{\ssigma}{\mathord}{sfletters}{"1B}
\DeclareMathSymbol{\stau}{\mathord}{sfletters}{"1C}
\DeclareMathSymbol{\supsilon}{\mathord}{sfletters}{"1D}
\DeclareMathSymbol{\sphi}{\mathord}{sfletters}{"1E}
\DeclareMathSymbol{\schi}{\mathord}{sfletters}{"1F}
\DeclareMathSymbol{\spsi}{\mathord}{sfletters}{"20}
\DeclareMathSymbol{\somega}{\mathord}{sfletters}{"21}

\DeclareFontEncoding{LGR}{}{}
\DeclareSymbolFont{sfgreek}{LGR}{cmss}{m}{n}
\SetSymbolFont{sfgreek}{bold}{LGR}{cmss}{bx}{n}
\DeclareMathSymbol{\ssalpha}{\mathord}{sfgreek}{`a}
\DeclareMathSymbol{\ssbeta}{\mathord}{sfgreek}{`b}
\DeclareMathSymbol{\ssgamma}{\mathord}{sfgreek}{`g}
\DeclareMathSymbol{\ssdelta}{\mathord}{sfgreek}{`d}
\DeclareMathSymbol{\ssepsilon}{\mathord}{sfgreek}{`e}
\DeclareMathSymbol{\sszeta}{\mathord}{sfgreek}{`z}
\DeclareMathSymbol{\sseta}{\mathord}{sfgreek}{`h}
\DeclareMathSymbol{\sstheta}{\mathord}{sfgreek}{`j}
\DeclareMathSymbol{\ssiota}{\mathord}{sfgreek}{`i}
\DeclareMathSymbol{\sskappa}{\mathord}{sfgreek}{`k}
\DeclareMathSymbol{\sslambda}{\mathord}{sfgreek}{`l}
\DeclareMathSymbol{\ssmu}{\mathord}{sfgreek}{`m}
\DeclareMathSymbol{\ssnu}{\mathord}{sfgreek}{`n}
\DeclareMathSymbol{\ssxi}{\mathord}{sfgreek}{`x}
\DeclareMathSymbol{\ssomicron}{\mathord}{sfgreek}{`o}
\DeclareMathSymbol{\sspi}{\mathord}{sfgreek}{`p}
\DeclareMathSymbol{\ssrho}{\mathord}{sfgreek}{`r}
\DeclareMathSymbol{\sssigma}{\mathord}{sfgreek}{`s}
\DeclareMathSymbol{\sstau}{\mathord}{sfgreek}{`t}
\DeclareMathSymbol{\ssupsilon}{\mathord}{sfgreek}{`u}
\DeclareMathSymbol{\ssphi}{\mathord}{sfgreek}{`f}
\DeclareMathSymbol{\sschi}{\mathord}{sfgreek}{`q}
\DeclareMathSymbol{\sspsi}{\mathord}{sfgreek}{`y}
\DeclareMathSymbol{\ssomega}{\mathord}{sfgreek}{`w}
\DeclareMathSymbol{\ssvarsigma}{\mathord}{sfgreek}{`c}
\DeclareMathSymbol{\ssGamma}{\mathalpha}{sfgreek}{`G}
\DeclareMathSymbol{\ssDelta}{\mathalpha}{sfgreek}{`D}
\DeclareMathSymbol{\ssTheta}{\mathalpha}{sfgreek}{`J}
\DeclareMathSymbol{\ssLambda}{\mathalpha}{sfgreek}{`L}
\DeclareMathSymbol{\ssXi}{\mathalpha}{sfgreek}{`X}
\DeclareMathSymbol{\ssPi}{\mathalpha}{sfgreek}{`P}
\DeclareMathSymbol{\ssSigma}{\mathalpha}{sfgreek}{`S}
\DeclareMathSymbol{\ssUpsilon}{\mathalpha}{sfgreek}{`U}
\DeclareMathSymbol{\ssPhi}{\mathalpha}{sfgreek}{`F}
\DeclareMathSymbol{\ssPsi}{\mathalpha}{sfgreek}{`Y}
\DeclareMathSymbol{\ssOmega}{\mathalpha}{sfgreek}{`W}

\DeclareRobustCommand{\msf}[1]{%
  \ifcat\noexpand#1\relax\msfgreek{#1}\else\mathsf{#1}\fi
}
\makeatletter
\newcommand{\msfgreek}[1]{\csname ss\expandafter\@gobble\string#1\endcsname}
\makeatother

\makeatletter
\def\moverlay{\mathpalette\mov@rlay}
\def\mov@rlay#1#2{\leavevmode\vtop{%
  \baselineskip\z@skip \lineskiplimit-\maxdimen
  \ialign{\hfil$\m@th#1##$\hfil\cr#2\crcr}}}
\newcommand{\charfusion}[3][\mathord]{
    #1{\ifx#1\mathop\vphantom{#2}\fi
        \mathpalette\mov@rlay{#2\cr#3}
      }
    \ifx#1\mathop\expandafter\displaylimits\fi}
\makeatother

\newcommand{\vect}[1]{\boldsymbol{\msf{#1}}}
\newcommand{\Mod}[1]{\ (\mathrm{mod}\ #1)}

\DeclareMathOperator{\odd}{odd}

\renewcommand{\epsilon}{\varepsilon}
\renewcommand{\ge}{\geqslant}
\renewcommand{\le}{\leqslant}
\renewcommand{\geq}{\geqslant}
\renewcommand{\leq}{\leqslant}

\DeclarePairedDelimiter{\abs}{\lvert}{\rvert}

\DeclarePairedDelimiter{\norm}{\lVert}{\rVert}
\DeclarePairedDelimiter{\set}{\{}{\}}

\newcommand{\defn}[1]{\textcolor{Maroon}{\emph{#1}}}

\newcommand*{\ints}{\mathbb{Z}}
\newcommand*{\rationals}{\mathbb{Q}}
\newcommand*{\reals}{\mathbb{R}}

\title{Exponential odd-distance sets under the Manhattan metric}

\author[Espuny D\'iaz]{Alberto Espuny D\'iaz}
\address[Espuny D\'iaz]{Institut f\"ur Informatik, Universit\"at Heidelberg, 69120 Heidelberg, Germany.}
\email{espuny-diaz@informatik.uni-heidelberg.de}

\author[Hogan]{Emma Hogan}
\address[Hogan, Michel]{Mathematical Institute, University of Oxford, United Kingdom.}
\email{\{hogan,michel\}@maths.ox.ac.uk}

\author[Illingworth]{Freddie Illingworth}
\address[Illingworth]{Department of Mathematics, University College London, United Kingdom.}
\email{f.illingworth@ucl.ac.uk \vspace{-8pt}}

\author[Michel]{Lukas Michel}

\author[Portier]{Julien Portier}
\address[Portier]{Department of Pure Mathematics and Mathematical Statistics (DPMMS), University of Cambridge, United Kingdom.}
\email{jp899@cam.ac.uk}

\author[Yan]{Jun Yan}
\address[Yan]{Mathematics Institute, University of Warwick, United Kingdom.}
\email{jun.yan@warwick.ac.uk}

\thanks{A.~Espuny Díaz was funded by the Deutsche Forschungsgemeinschaft (DFG, German Research Foundation) through project no.\ 513704762. E.~Hogan's research was supported by EPSRC grant EP/W524311/1. F.~Illingworth's research was supported by EPSRC grant EP/V521917/1 and the Heilbronn Institute for Mathematical Research. J.~Yan was supported by the Warwick Mathematics Institute CDT and funding from the EPSRC grant EP/W523793/1.}

\date{\today}

\begin{document}

\begin{abstract}
We construct a set of $2^n$ points in $\mathbb{R}^n$ such that all pairwise Manhattan distances are odd integers, which improves the recent linear lower bound of Golovanov, Kupavskii and Sagdeev. In contrast to the Euclidean and maximum metrics, this shows that the odd-distance set problem behaves very differently to the equilateral set problem under the Manhattan metric. Moreover, all coordinates of the points in our construction are integers or half-integers, and we show that our construction is optimal under this additional restriction.
\end{abstract}


\maketitle

\section{Introduction}\label{sec:intro}

An \defn{equilateral set} in a metric space is a set of points in which all pairwise distances are the same. Determining the largest equilateral set in a metric space has been well-studied and the case of $\reals^n$ equipped with the $\ell_p$-norm\footnote{A point $\vect{v} = (v_1, \dotsc, v_n) \in \reals^n$ has $\ell_p$-norm $\norm{\vect{v}}_p \coloneqq \bigl(\sum_i \abs{v_i}^p\bigr)^{1/p}$ for $1 \leq p < \infty$ and has $\ell_\infty$-norm $\norm{\vect{v}}_\infty \coloneqq \max_i \abs{v_i}$.} $\norm{\cdot}_p$ has received significant attention. Denoting the size of the largest equilateral set in $\reals^n$ with the $\ell_p$-norm by $e_p(n)$, the following table gives a summary of the best bounds known for $p \in \{ 1, 2, \infty \}$.

\begin{table}[H]
    \centering
    \begin{tabular}{lllll}
        \toprule
        & \multicolumn{2}{l}{Lower bound} & \multicolumn{2}{l}{Upper bound} \\
        \midrule
        $e_1(n)$ & $2n$ & cross polytope & $cn \log n$ & \citet{AP03} \\
        $e_2(n)$ & $n + 1$ & unit simplex & $n + 1$ & folklore \\
        $e_\infty(n)$ & $2^n$ & unit hypercube & $2^n$ & \citet{P71} \\
        \bottomrule
    \end{tabular}
    \caption{The size of the largest equilateral sets in $(\reals^n, \norm{\cdot}_p)$.}
\end{table}

An old conjecture of Kusner~\cite{Kus83} states that $e_p(n) = n + 1$ for all integers $1 < p < \infty$ and $e_1(n) = 2n$. For progress on Kusner's conjecture, see for example \cite{AP03,Swa04,Smy13}.

\Citet{GRS74} introduced the following relaxation of an equilateral set. An \defn{odd-distance set} in a metric space is a set of points in which all pairwise distances are odd integers. We denote the size of the largest odd-distance set in $\reals^n$ equipped with the $\ell_p$-norm by $\odd_p(n)$. In a normed space, any equilateral set can be scaled to a set where all pairwise distances are $1$, and so $\odd_p(n) \geq e_p(n)$. The following table gives a summary of the best bounds known for $\odd_p(n)$. The bounds for $\odd_2(n)$ were proved in the original paper of \citet{GRS74}, who in fact determined the exact value of $\odd_2(n)$ (it is $n + 1$ except when $n \equiv -2 \pmod{16}$). The other bounds are due to a recent paper of \citet{GKS23}.

\begin{table}[H]
    \centering
    \begin{tabular}{lllll}
        \toprule
        & \multicolumn{2}{l}{Lower bound} & \multicolumn{2}{l}{Upper bound} \\
        \midrule
        $\odd_1(n)$ & $(\frac{7}{3} - o(1))n$  & \cite{GKS23} & $(4 + o(1)) n! \, n \log n$ & \cite{GKS23} \\
        $\odd_2(n)$ & $n + 1$ & unit simplex & $n + 2$ & \cite{GRS74} \\
        $\odd_\infty(n)$ & $2^n$ & unit hypercube & $2^n$ & \cite{GKS23} \\
        \bottomrule
    \end{tabular}
    \caption{The size of the largest odd-distance sets in $(\reals^n, \norm{\cdot}_p)$.}
\end{table}

From these results we see that relaxing from the equilateral to the odd-distance problem makes very little difference for the $\ell_2$ and $\ell_\infty$-norms. The case of the $\ell_1$-norm is less clear: the lower bound of \citet{GKS23} for $\odd_1(n)$ is larger than Kusner's conjectured upper bound for $e_1(n)$, but only by a constant multiplicative factor, and it is still smaller than the upper bound proven by \citet{AP03}.

In this note, we significantly improve the lower bound for $\odd_1(n)$, showing that there are odd-distance sets of size $2^n$ under the $\ell_1$-norm. Together with the upper bound for $e_1(n)$ of \citet{AP03}, this confirms in a strong sense that the equilateral and odd-distance problems behave very differently for the $\ell_1$-norm.

\begin{restatable}{theorem}{mainthm}\label{thm:main}
    For all positive integers\/ $n$,\/ $\odd_1(n) \geq 2^n$.
    Moreover, there is an $\ell_1$-odd-distance set $\cP \subseteq (\frac{1}{2} \cdot \ints)^n$ of size\/ $2^n$.
\end{restatable}

It is not possible to replace the half-integer lattice $(\frac{1}{2} \cdot \ints)^n$ by the integer lattice $\ints^n$ in the statement of \cref{thm:main}.
Indeed, there is not even an odd-distance set of size $3$ in $\ints^n$: every integer~$x$ satisfies $\abs{x} \equiv x \pmod{2}$ and so, for any integer points $\vect{a}, \vect{b}, \vect{c}\in\ints^n$, we have that $\norm{\vect{a} - \vect{b}}_1 + \norm{\vect{b} - \vect{c}}_1 \equiv \norm{\vect{a} - \vect{c}}_1 \pmod{2}$.

We further show that \cref{thm:main} is best possible in the case that all coordinates are half-integers.

\begin{restatable}{proposition}{bestprop}\label{prop:best}
    Let\/ $\cP \subseteq (\frac{1}{2}\cdot\ints)^n$ be an\/ $\ell_1$-odd-distance set.
    Then,\/ $\abs{\cP} \leq 2^n$.
\end{restatable}

As a last result, we also note that, when considering odd-distance sets in the $\ell_1$-norm, we may restrict the coordinates to be rational numbers whose denominator is a power of $2$.

\begin{restatable}{proposition}{poweroftwoprop}\label{prop:powerof2}
    Let\/ $\cT \subseteq \rationals$ be the set of rational numbers whose denominators are powers of\/~$2$.
    For any\/ $\ell_1$-odd-distance set\/ $\cP \subseteq \reals^n$, there exists an\/ $\ell_1$-odd-distance set\/ $\cQ \subseteq \cT^n$ of the same size as $\cP$.
\end{restatable}

\section{Proofs}

To prove the lower bound on $\odd_1(n)$, we will inductively construct an odd-distance set of size~$2^n$. We will achieve this by splitting in two a coordinate of each point of an odd-distance set in~$\reals^n$. By doing this in two distinct ways, we replace each point with two points in $\reals^{n+1}$. This construction relies on the following simple lemma.

\begin{lemma}\label{lem:dim2}
    For every\/ $x \in \frac{1}{2} \cdot \ints$, there exists an odd-distance set\/ $\cP\subseteq(\frac{1}{2} \cdot \ints)^2$ of size $2$ such that the coordinates of each point sum to\/ $x$ and all coordinates are at least\/ $x/2 - 1/2$ and at most\/ $x/2 + 1/2$.
\end{lemma}

\begin{proof}
    If $x\in\ints$, let
    \begin{equation*}
        \cP = \set[\bigg]{\biggl(\frac{x}{2}+\frac{1}{2}, \frac{x}{2} - \frac{1}{2}\biggr), \biggl(\frac{x}{2},\frac{x}{2}\biggr)}.
    \end{equation*}
    Otherwise, if $x \in \ints + \frac{1}{2}$, let
    \begin{equation*}
        \cP = \set[\bigg]{\biggl(\frac{x}{2} + \frac{1}{4}, \frac{x}{2} - \frac{1}{4}\biggr), \biggl(\frac{x}{2} - \frac{1}{4}, \frac{x}{2} + \frac{1}{4}\biggr)}.
    \end{equation*}
    Clearly, in both cases $\cP$ satisfies the claim.
\end{proof}

We may now prove our main result. Our strategy is to start with an odd-distance set \mbox{$\cP \subseteq (\frac{1}{2} \cdot \ints)^n$} and first space out the points in $\cP$ in the first dimension. We achieve this by translating the first coordinate of each point by some even integer while maintaining their order in the first dimension and the odd-distance property. Then, we replace each first coordinate by the odd-distance set in $(\frac{1}{2} \cdot \ints)^2$ given by \cref{lem:dim2}.
Because the first dimension was sufficiently spaced out in advance, this guarantees that the relative order of the points in the two new dimensions will be the same as their relative order in the old dimension.
As a result, the new set will be an odd-distance set in $(\frac{1}{2} \cdot \ints)^{n + 1}$ of size $2 \abs{\cP}$.

\mainthm*
\begin{proof}
    We will prove, by induction on $n$, that $(\frac{1}{2} \cdot \ints)^n$ contains an $\ell_1$-odd-distance set of size~$2^n$ for all $n\geq1$.
    For $n = 1$, the set $\cP = \set{0, 1}$ suffices.
    Now, suppose~$\cP$ is an odd-distance set in $(\frac{1}{2} \cdot \ints)^n$ of size $2^n$.
    We label the points of $\cP$ as $\vect{p}_1$, \ldots, $\vect{p}_{2^n}$ so that the first coordinate is increasing.
    That is, for all $i\in[2^n]$ we may write $\vect{p}_i = (a_i, \vect{v}_i)$, where $a_1 \leq a_2 \leq \dotsb \leq a_{2^n}$ and the vectors $\vect{v}_i$ are in $(\frac{1}{2} \cdot \ints)^{n - 1}$.

    For each $i\in[2^n]$, let $b_i \coloneqq a_i + 2i$ and $\vect{q}_i \coloneqq (b_i, \vect{v}_i)$.
    Note that $\cQ \coloneqq \set{\vect{q}_1, \dotsc, \vect{q}_{2^n}}$ is also an $\ell_1$-odd-distance set in $(\frac{1}{2} \cdot \ints)^n$.
    Indeed, for $1\leq i < j \leq 2^n$,
    \begin{align*}
        \norm{\vect{q}_j - \vect{q}_i}_1 & = (b_j - b_i) + \norm{\vect{v}_j - \vect{v}_i}_1 \\
        & = (a_j - a_i) + \norm{\vect{v}_j - \vect{v}_i}_1 + 2(j - i) \\
        & = \norm{\vect{p}_j - \vect{p}_i}_1 + 2(j - i),
    \end{align*}
    which is odd.

    By \cref{lem:dim2}, for each $i\in[2^n]$ there is an odd-distance set $\cC_i = \set{\vect{c}^{(1)}_i, \vect{c}^{(2)}_i} \subseteq (\frac{1}{2} \cdot \ints)^2$ where the coordinates of each point sum to $b_i$ and all coordinates are between \mbox{$b_i/2 - 1/2$} and \mbox{$b_i/2 + 1/2$}.
    For each $i\in[2^n]$, let $\vect{q}_i^{(1)} \coloneqq (\vect{c}_i^{(1)}, \vect{v}_i)$ and $\vect{q}_i^{(2)} \coloneqq (\vect{c}_i^{(2)}, \vect{v}_i)$.
    Finally, let $\cQ' \coloneqq \set{\vect{q}_1^{(1)}, \vect{q}_1^{(2)}, \dotsc, \vect{q}_{2^n}^{(1)}, \vect{q}_{2^n}^{(2)}}$.
    Clearly, since all the $b_i$ are distinct, $\cQ'$ is a subset of $(\frac{1}{2} \cdot \ints)^{n + 1}$ of size $2^{n + 1}$.
    It remains to show that it is an $\ell_1$-odd-distance set.

    First note that, for all $i\in[2^n]$,
    \begin{equation*}
        \norm{\vect{q}_i^{(1)} - \vect{q}_i^{(2)}}_1 = \norm{\vect{c}_i^{(1)} - \vect{c}_i^{(2)}}_1 + \norm{\vect{v}_i - \vect{v}_i}_1 = \norm{\vect{c}_i^{(1)} - \vect{c}_i^{(2)}}_1,
    \end{equation*}
    which is odd as $\cC_i$ is an odd-distance set.
    Now fix $1\leq i < j \leq 2^n$ and consider the distance between $\vect{q}_i^{(r)}$ and $\vect{q}_j^{(s)}$, where $r, s \in \set{1, 2}$.
    Recall, from our choice of the $b_i$, that $b_j \geq b_i + 2$.
    Since the coordinates of $\vect{c}_i^{(r)}$ are at most $b_i/2 + 1/2$ and the coordinates of $\vect{c}_j^{(s)}$ are at least $b_j/2 - 1/2$, both coordinates of $\vect{c}_i^{(r)}$ are bounded from above by both coordinates of $\vect{c}_j^{(s)}$.
    Since the sum of the coordinates of $\vect{c}_i^{(r)}$ is $b_i$ and the sum of the coordinates of $\vect{c}_j^{(s)}$ is $b_j$, it follows that
    \begin{equation*}
        \norm{\vect{c}_j^{(s)} - \vect{c}_i^{(r)}}_1 = b_j - b_i.
    \end{equation*}
    Thus,
    \begin{equation*}
        \norm{\vect{q}_j^{(s)} - \vect{q}_i^{(r)}}_1 = \norm{\vect{c}_j^{(s)} - \vect{c}_i^{(r)}}_1 + \norm{\vect{v}_j - \vect{v}_i}_1 = (b_j - b_i) + \norm{\vect{v}_j - \vect{v}_i}_1 = \norm{\vect{q}_j - \vect{q}_i}_1,
    \end{equation*}
    which is odd since $\cQ$ is an odd-distance set, as required.
\end{proof}

If we restrict a point set to having half-integer coordinates, we show that our construction is optimal.
We prove this with a simple pigeonhole argument.

\bestprop*
\begin{proof}
    Let $\phi \colon \cP \to \set{0, 1}^n$ be defined by setting $\phi(\vect{p}) \coloneqq (\mathds{1}_{\set{p_i \in \ints + 1/2}})_{i \in [n]}$.

    First, we note that there cannot be two points $\vect{p}, \vect{q} \in \cP$ with $\norm{\phi(\vect{p})}_1 \nequiv \norm{\phi(\vect{q})}_1 \Mod{2}$.
    Indeed, in such a case, there would be an odd number of coordinates $i\in[n]$ such that $\abs{p_i-q_i}\in\ints+1/2$, which implies that $\norm{\vect{p}-\vect{q}}_1$ is not an integer.
    In particular, $\abs{\phi(\cP)} \leq 2^{n-1}$.

    Secondly, we claim that there cannot be three points $\vect{a}, \vect{b}, \vect{c}\in\cP$ with $\phi(\vect{a}) = \phi(\vect{b}) = \phi(\vect{c})$.
    Indeed, otherwise $a_i - b_i$, $b_i - c_i$, and $c_i - a_i$ are integers for all $i\in[n]$, and so
    \begin{align*}
        \norm{\vect{a} - \vect{b}}_1 = \sum_{i=1}^n \abs{a_i - b_i} & \equiv\sum_{i=1}^n (a_i - b_i) \pmod{2}, \\
        \norm{\vect{b} - \vect{c}}_1 = \sum_{i=1}^n \abs{b_i - c_i} & \equiv \sum_{i=1}^n (b_i - c_i) \pmod{2}, \\
        \norm{\vect{a} - \vect{c}}_1 = \sum_{i=1}^n \abs{a_i - c_i} & \equiv \sum_{i=1}^n (a_i - c_i) \pmod{2}.
    \end{align*}
    Adding the first two expressions and comparing with the last one, we conclude that
    \begin{equation*}
        \norm{\vect{a} - \vect{b}}_1 + \norm{\vect{b}-\vect{c}}_1\equiv\norm{\vect{a}-\vect{c}}_1\pmod{2},
    \end{equation*}
    which is impossible, as the sum of two odd numbers is even.
    Thus, $\abs{\cP} \leq 2 \cdot \abs{\phi(\cP)} \leq 2^n$.
\end{proof}

Finally, we show that we may assume all points in $\ell_1$-odd-distance sets $\cP$ have rational coordinates.
Indeed, the pairwise distances between the points of $\cP$ can be expressed by a system of linear equations on the coordinates of the points of $\cP$.
Solutions to this system of linear equations can be characterised by a set of free variables whose choice determines all other variables.
If we now replace all coordinates corresponding to free variables by nearby rational numbers, this determines all remaining coordinates, and those coordinates will also be rational since the coefficients of the linear equations were rational.
This gives a point set with rational coefficients, and if the order of the coordinates did not change, then the distances between the points remain the same.
So, the point set will be an odd-distance set.

\begin{lemma}\label{lem:rational}
For any $\ell_1$-odd-distance set\/ $\cP\subseteq\reals^n$, there exists an odd-distance set $\cQ\subseteq\rationals^n$ of the same size as $\cP$. 
\end{lemma}

\begin{proof}
    By iteratively translating coordinates as in the proof of \cref{thm:main}, we may assume that \mbox{$\abs{p_i - q_i} \ge 2$} for all distinct points $\vect{p}, \vect{q} \in \cP$ and all $i\in[n]$.
    Let $s_i^{\vect{p},\vect{q}} \coloneqq 1$ if $p_i > q_i$ and $s_i^{\vect{p},\vect{q}} \coloneqq -1$ otherwise.
    Then,
    \[
        \norm{\vect{p} - \vect{q}}_1 = \sum_{i = 1}^n s_i^{\vect{p},\vect{q}} (p_i - q_i).
    \]
    Write $d_{\vect{p},\vect{q}} \coloneqq \norm{\vect{p} - \vect{q}}_1$ and introduce a variable $x_{\vect{p},i}$ for all $\vect{p} \in \cP$ and all $i\in[n]$.
    This means that the system of linear equations
    \[
        \sum_{i = 1}^n s_i^{\vect{p},\vect{q}} \left(x_{\vect{p},i} - x_{\vect{q},i}\right) = d_{\vect{p},\vect{q}}
    \]
    for all distinct $\vect{p}, \vect{q} \in \cP$ has a solution given by $x_{\vect{p},i} = p_i$ for all $\vect{p} \in \cP$ and all $i\in[n]$.
    
    Since all coefficients and all $d_{\vect{p},\vect{q}}$ are integers, the reduced row echelon form of this system of linear equations only has rational coefficients.
    Thus, in the reduced row echelon form, there exists a set of indices $\cF \subseteq \cP \times [n]$ such that $x_f$ is a free variable for all $f \in \cF$ while $x_g$ is a dependent variable for all $g \notin \cF$, with $x_g = a_g + \sum_{f \in \cF} b_{f,g} x_f$ for some $a_g, b_{f,g} \in \rationals$.

    Let $C \coloneqq \max_{f \in \cF, g \notin \cF} \abs{b_{f,g}}$.
    For each $(\vect{p},i) \in \cF$, choose $y_{\vect{p},i} \in \rationals$ so that $\abs{y_{\vect{p},i} - p_i} < 1/(C \abs{\cP} n)$.
    For each $(\vect{p},i) \notin \cF$ let $y_{\vect{p},i} = a_{\vect{p},i} + \sum_{f \in \cF} b_{f,(\vect{p},i)} y_f$, and note that $y_{\vect{p},i} \in \rationals$ and
    \begin{align*}
        \abs{y_{\vect{p},i} - p_i} & = \abs*{\left(a_{\vect{p},i} + \sum_{(\vect{q},j) \in \cF} b_{(\vect{q},j),(\vect{p},i)} y_{\vect{q},j}\right) - \left(a_{\vect{p},i} + \sum_{(\vect{q},j) \in \cF} b_{(\vect{q},j),(\vect{p},i)} q_j\right)} \\
        & \le \sum_{(\vect{q},j) \in \cF} \abs*{b_{(\vect{q},j),(\vect{p},i)} \left(y_{\vect{q},j} - q_j\right)} < \frac{\abs{\cF} C}{C \abs{\cP} n} \le 1,
    \end{align*}
    where the first equality used the fact that $x_{\vect{p},i} = p_i$ is a solution to the system of equations.
    
    For each $\vect{p} \in \cP$, let $\vect{p}' \in \rationals^n$ be the point with $p'_i = y_{\vect{p},i}$.
    Consider two distinct points $\vect{p}, \vect{q} \in \cP$.
    For each $i$, if $p_i < q_i$, then $p_i + 2 \le q_i$ by assumption, and since $\abs{p'_i - p_i} < 1$ and $\abs{q'_i - q_i} < 1$, this implies that $p'_i < q'_i$.
    Similarly, if $p_i > q_i$, then $p'_i > q'_i$.
    Therefore,
    \[
        \norm{\vect{p}' - \vect{q}'}_1 = \sum_{i = 1}^n s_i^{\vect{p},\vect{q}} (p'_i - q'_i) = d_{\vect{p},\vect{q}},
    \]
    where the last equality uses the fact that $x_{\vect{p},i} = p'_i$ is a solution to the system of equations.
    In particular, the distance between $\vect{p}'$ and $\vect{q}'$ is odd, and so $\cQ = \set{\vect{p}' \colon \vect{p} \in \cP} \subseteq \rationals^n$ is an odd-distance set of the same size as $\cP$.
\end{proof}

\Cref{prop:powerof2} is now an easy consequence since, if we scale all coordinates of an odd-distance set by a fixed odd integer, the resulting set is still an odd-distance set.

\poweroftwoprop*
\begin{proof}
By \cref{lem:rational}, there exists an odd-distance set $\cQ\subseteq\rationals^n$ of the same size as $\cP$. Let $C$ be the least common multiple of all odd integers that divide the denominator of any coordinate of any point of $\cQ$.
Clearly, $C$ is an odd integer, and $\cQ' \coloneqq \set{C \vect{q} \colon \vect{q} \in \cQ} \subseteq \cT^n$.
Moreover, for all distinct points $C \vect{p}, C \vect{q} \in \cQ'$ we have $\norm{C \vect{p} - C \vect{q}}_1 = C \norm{\vect{p} - \vect{q}}_1$, which is an odd integer as a product of two odd integers, and so $\cQ'$ is an odd-distance set.
\end{proof}

\section*{Acknowledgement}

The research leading to the results in this note began during the \emph{Early Career Researchers in Combinatorics} (ECRiC) workshop, 15-19 July 2024, funded by the International Centre for Mathematical Sciences, Edinburgh.

\end{document}